%% file: PolyhedralCones.tex
\documentclass[11pt,reqno]{amsart}
\usepackage{tikz}
\usetikzlibrary{calc,through,backgrounds,decorations.pathreplacing,decorations.pathmorphing,shapes.geometric}
\usepackage[utf8]{inputenc}
\usepackage[english]{babel}
\usepackage{amsmath}
\usepackage{amssymb}
\usepackage{bbm}
\usepackage{latexsym}
\usepackage{a4wide}
\newcommand{\captionfonts}{\Small}
\makeatletter
\long\def\@makecaption#1#2{%
  \vskip\abovecaptionskip
  \sbox\@tempboxa{{\captionfonts #1: #2}}%
  \ifdim \wd\@tempboxa >\hsize
    {\captionfonts #1: #2\par}
  \else
    \hbox to\hsize{\hfil\box\@tempboxa\hfil}%
  \fi
  \vskip\belowcaptionskip}
\makeatother
\usepackage{graphicx}

\DeclareMathOperator{\conv}{conv}
\DeclareMathOperator{\ccone}{ccone}

\newcommand{\R}{\mathbb{R}}

\newcommand{\Q}{\mathbb{Q}}

\newcommand{\scalProd}[2]{\langle{#1},{#2}\rangle}
\newcommand{\setDef}[2]{\{{#1}\,:\,{#2}\}}

\newcommand{\zeroVec}[1]{\mathbb{O}}
\newcommand{\oneVec}[1]{\mathbbm{1}}

\newcommand{\sd}[1]{\delta({#1})}

\newcommand{\qsd}[1]{\Delta({#1})}
\DeclareMathOperator{\polyOp}{P}
\newcommand{\polyLe}[2]{\polyOp^{\le}({#1},{#2})}
\DeclareMathOperator{\kernelOp}{ker}
\newcommand{\kernel}[1]{\kernelOp({#1})}

\newcommand{\norm}[1]{\|{#1}\|}

\DeclareMathOperator{\transposeOp}{T}
\newcommand{\transpose}[1]{{#1}^{\transposeOp}}

\newcommand{\row}[2]{{#1}_{{#2},\star}}

\newcommand{\st}{^{\star}}
\DeclareMathOperator{\NP}{NP}

\newtheorem{theorem}{Theorem}
\newtheorem{lemma}{Lemma}

\title{Another Proof of the Fact that Polyhedral Cones are Finitely Generated}
\author{Volker Kaibel}
\email{kaibel@ovgu.de}%
\address{Otto-von-Guericke Universit\"at Magdeburg, Fakult\"at f\"ur Mathematik, Universit\"atsplatz~2, 39106~Magdeburg, Germany}
\date{\today}

\begin{document}

\begin{abstract}
	In this note, we work out a simple inductive proof showing that every polyhedral cone~$K$ is the conic hull of a finite set~$X$ of vectors. The base cases of the induction are linear subspaces and linear halfspaces of linear subspaces. The proof also shows that the components of the vectors in~$X$ can be chosen (up to their sign) to be quotients of subdeterminants of the coefficient matrix of any inequality system defining~$K$.
\end{abstract}

\maketitle


A matrix~$A\in\R^{m\times n}$ and a vector~$b\in\R^m$ define the \emph{polyhedron}
	$\polyLe{A}{b}=\setDef{x\in\R^n}{Ax\le b}$. 
A polyhedron $\polyLe{A}{\zeroVec{}}$ is a \emph{polyhedral cone}. A \emph{polytope} is the \emph{convex hull}
\begin{equation*}
	\conv(X)=\setDef{\sum_{x\in X}\lambda_xx}{\lambda_x\in\R, \lambda_x\ge 0\text{ for all }x\in X,\sum_{x\in X}\lambda_x=1}
\end{equation*}
of a finite set $X\subseteq\R^n$, and a
 \emph{finitely generated cone} is the (\emph{convex}) \emph{conic hull}
\begin{equation*}
	\ccone(X)=\setDef{\sum_{x\in X}\lambda_xx}{\lambda_x\in\R, \lambda_x\ge 0\text{ for all }x\in X}
\end{equation*}
of  a finite set~$X\subseteq\R^n$. A classical theorem (which is at the core of the theory of polyhedra) due to Weyl~\cite{Wey35} and Minkowski~\cite{Min1896} states that a subset of~$\R^n$ is a polyhedron if and only if it is the \emph{Minkowski sum} ($S+T=\setDef{s+t}{s\in S,t\in T}$ for $S,T\subseteq \R^n$) of a polytope and a finitely generated cone. 

A representation $P=\polyLe{A}{b}$ (with $A\in\R^{m\times n}$, $b\in\R^m$) of a polyhedron~$P\subseteq\R^n$ is called an \emph{outer description}, while $P=\conv(V)+\ccone(W)$ with finite sets $V,W\subseteq \R^n$ is an \emph{inner description}. Later refinements (which are very important for the theory of linear and integer programming) state that, given one representation of a polyhedron~$P$, there is a representation of~$P$ of the other type all of whose components are quotients of determinants of matrices (of size~$n\times n$) formed from components of the given representation. In particular, if the given representation of~$P$ is rational, then one can choose a rational representation of~$P$ of the other type all of whose components have an
 encoding length (say, in the binary system) that is bounded polynomially in~$n$ and the maximal encoding length of any component of the given representation (see, e.g., \cite[Thm.~10.2]{Sch86}). This is not only necessary for the polynomial solvability of the linear programming problem, but it is also crucial for establishing that the integer programming feasibilty problem is contained in the complexity class~$\NP$ (see, e.g., \cite[Cor.~17.1b]{Sch86}).

We denote by~$\sd{M}$ the set of all determinants of submatrices (formed by  arbitrary subsets of rows and columns of equal cardinality, including the empty submatrix, whose determinant is considered to be one) of a matrix $M\in\R^{m\times n}$, and define
\begin{equation*}
	\qsd{M}=\setDef{\tfrac{p}{q}}{p,q\in \sd{M}\cup (-\sd{M}), q\ne 0}\,.
\end{equation*}
Clearly, for \emph{rational} matrices $M\in\Q^{m\times n}$, we have $\qsd{M}\subseteq\Q$. It is well-known that, using the concepts of \emph{homogenization} and \emph{polarity},
one can easily derive 
 Weyl's and Minkowski's theorem, including the above mentioned refinements, from the following result (which is itself a special case of one direction of the Weyl-Minkowski theorem). 

\begin{theorem}\label{thm:polyConeFinite}
	For every matrix $A\in\R^{m\times n}$, there is a finite set $X\subseteq\qsd{A}^n$ with 
	\begin{equation*}
		\polyLe{A}{\zeroVec{}}=\ccone(X)\,.
	\end{equation*}
\end{theorem}

The proof of Theorem~\ref{thm:polyConeFinite}  we are going to work out  imitates the following ``obvious'' inductive proof of the similar statement that every bounded polyhedron~$P$ is a polytope (see~\cite{Mat02}): If~$P$ consists of one point only, then the statement is clear. Otherwise, the boundary of~$P$ is the union of finitely many lower dimensional bounded polyhedra, its \emph{faces}, which are polytopes by induction. The union of finite generating sets of these polytopes yields a generating set for~$P$, since for every~$x$ in the \emph{bounded} polyhedron~$P$, any line containing~$x$ and another point from~$P$ will intersect  the boundary of~$P$ in two points of which~$x$ is a convex combination.

There are two issues to deal with in order to turn this basic idea into an elementary proof of Theorem~\ref{thm:polyConeFinite}. First, we do not want to dwell
on the geometric concepts of faces and of the dimension of a polyhedron. This can indeed easily be avoided by allowing equations in the system defining the polyhedral cone and simply basing the induction on the number of inequalities. Second, and more important, in case of polyhedral cones~$K$ (instead of bounded polyhedra) it is, in general, not true that through every $y\in K$ there is a line intersecting the boundary of~$K$ in two points (of which~$y$ is a convex combination), as one can easily see at the examples of~$K$ being a linear subspace of~$\R^n$ or a linear halfspace of such a subspace. The core of the proof of Theorem~\ref{thm:polyConeFinite} presented here is to show that these are the only inconvenient cases. This is the essence of the  following  lemma, where $\kernel{M}=\setDef{x\in\R^n}{Mx=\zeroVec{}}$ denotes the \emph{kernel} of the matrix $M\in\R^{m\times n}$ (with $\kernel{M}=\R^n$ in case of $m=0$). 

\begin{lemma}\label{lem:GeomKonv:polyKegel:HilfeVA}
	Let $B\in\R^{p\times n}$, $C\in\R^{q\times n}$ (with $p+q\ge 1$, $n\ge 1$), 
	$A=
	\scriptsize{\begin{pmatrix}
		B \\ C
	\end{pmatrix}}\in\R^{(p+q)\times n}$,
	and
	$K=\setDef{x\in\R^n}{Bx\le \zeroVec{p},Cx=\zeroVec{q}}$.
	\begin{enumerate}
		\item[(i)] If we have $\dim(\kernel{B}\cap\kernel{C})\ge\dim(\kernel{C})-1$, then there is a finite set 
		$X\subseteq \qsd{A}^n$
		 satisfying 
		$K=\ccone(X)$.
		\item[(ii)] Otherwise, there is some vector $z\in\kernel{C}\setminus\{\zeroVec{n}\}$ with
		$z\not\in K$ and 	$-z\not\in K$.
	\end{enumerate}
\end{lemma}

\begin{proof}
	For the proof of the first part, 
	suppose that 
	\[
	U=\kernel{B}\cap\kernel{C}\quad\subseteq \quad K\quad \subseteq\quad\kernel{C}
	\]
	has dimension $\dim(U)\ge\dim(\kernel{C})-1$. 
	Let $\mathcal{B}'\subseteq\qsd{A}^n$ be a basis of $\kernel{C}$ for which $\mathcal{B}=\mathcal{B}'\cap U$ 
	is a basis of~$U$; due to Cramer's rule (see, e.g., \cite[Cor.~3.1c]{Sch86}), we can choose $\mathcal{B}'\subseteq\qsd{A}^n$. 
	We have
	$U=\ccone(\mathcal{B}\cup(-\mathcal{B}))$,
	and we may assume
    $U\subsetneq K$, since otherwise the claim clearly follows with $X=\mathcal{B}\cup(-\mathcal{B})$. In particular, we have $\dim(U)=\dim(\kernel{C})-1$. Hence, there is some  $a\in\kernel{C}\setminus U$ satisfying
	$U=\setDef{x\in\kernel{C}}{\scalProd{a}{x}=0}$ and 
	$\norm{a}=1$. Due to  dimension reasons, $\kernel{C}$ is the linear subspace of $\R^n$~generated by $U\cup\{a\}$. Therefore, for each $y\in\kernel{C}\setminus U$ (with $\scalProd{a}{y}\ne 0$),
	there is some $u\in U\subseteq K$ satisfying (see Fig.~\ref{fig:Lem1a})
	\begin{equation}\label{eq:GeomKonv:polyKegel:HilfeVA:y}
		y=u+\scalProd{a}{y}a	
	\end{equation}
	and thus
	\begin{equation}\label{eq:GeomKonv:polyKegel:HilfeVA:a}
	a=u'+\frac{1}{\scalProd{a}{y}}y
	\end{equation}
	with $u'=\tfrac{-1}{\scalProd{a}{y}}u\in U\subseteq K$.
	
	\begin{figure}[ht]
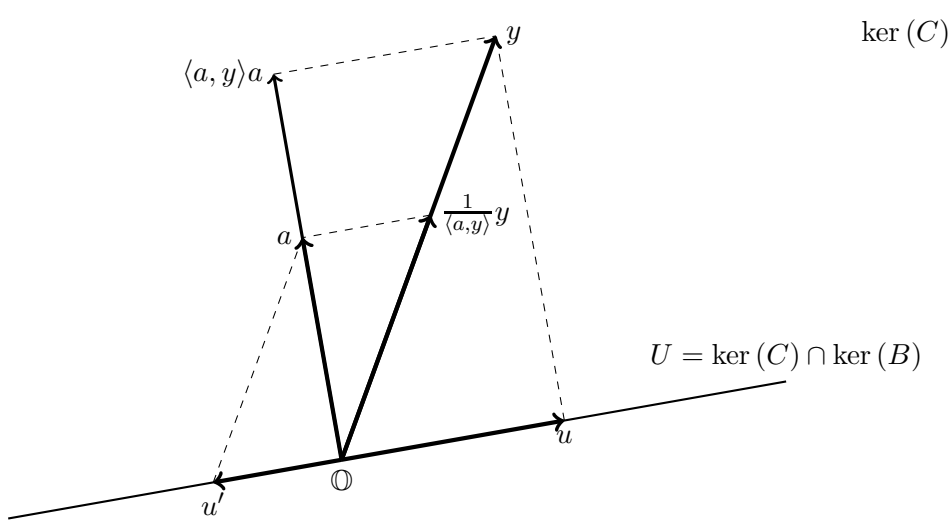

		\centering
		\input Fig-Lem1a
		\caption{Illustration of relations~\eqref{eq:GeomKonv:polyKegel:HilfeVA:y} and~\eqref{eq:GeomKonv:polyKegel:HilfeVA:a}.}
		\label{fig:Lem1a}
	\end{figure}

	After  possibly replacing~$a$ by~$-a$ we have $\scalProd{a}{y}>0$ for an arbitrarily chosen $y\in K\setminus U\ne\varnothing$,  which by~\eqref{eq:GeomKonv:polyKegel:HilfeVA:a} implies $a\in K$. Thus, 
	\begin{equation}\label{eq:GeomKonv:polyKegel:HilfeVA:aTuts}
		K=\ccone(U\cup\{a\})
	\end{equation}
	holds, where ``$\supseteq$'' is clear, and, in order to prove the reverse inclusion, due to~\eqref{eq:GeomKonv:polyKegel:HilfeVA:y}, it suffices to establish~$\scalProd{a}{y}>0$ for all $y\in K \setminus U$ as follows. For each $y\in K \setminus U$ we have $\scalProd{a}{y}\ne 0$, where  $\scalProd{a}{y}<0$ yields $\tfrac{1}{\scalProd{-a}{y}}>0$, hence $-a\in K$ by relation~\eqref{eq:GeomKonv:polyKegel:HilfeVA:a} (since $u'\in U$ implies $-u'\in U\subseteq K$, as~$U$ is a linear subspace). But from $a,-a\in K$ one deduces $Ba\le\zeroVec{}$ and $-Ba\le\zeroVec{}$, thus $a\in\kernel{B}$, contradicting $a\not\in U$.
		
	Due to 
	 $\dim(U)=\dim(\kernel{C})-1$ there is exactly one vector~$v\in\mathcal{B}'\setminus\mathcal{B}$ in the  basis~$\mathcal{B}'$ of~$\kernel{C}$ that is not contained in the basis~$\mathcal{B}$ of~$U$.
	We can assume  $\scalProd{a}{v}>0$ (by possibly replacing~$v$ by~$-v$). Thus, by~\eqref{eq:GeomKonv:polyKegel:HilfeVA:y} and~\eqref{eq:GeomKonv:polyKegel:HilfeVA:a} (with $y=v$)  we have
	$\ccone(U\cup\{a\})=\ccone(U\cup\{v\})$,
	which establishes the claim in the first part (because of~\eqref{eq:GeomKonv:polyKegel:HilfeVA:aTuts}) with $X=\mathcal{B}\cup\{v\}$.
		
	\smallskip
	For the proof of the second part of the lemma, let $L\subseteq\R^n$ be the linear space generated by the sum~$\transpose{B}\oneVec{}$ of all rows of~$B$ (see Fig.~\ref{fig:Lem1b}). 
	\begin{figure}[ht]
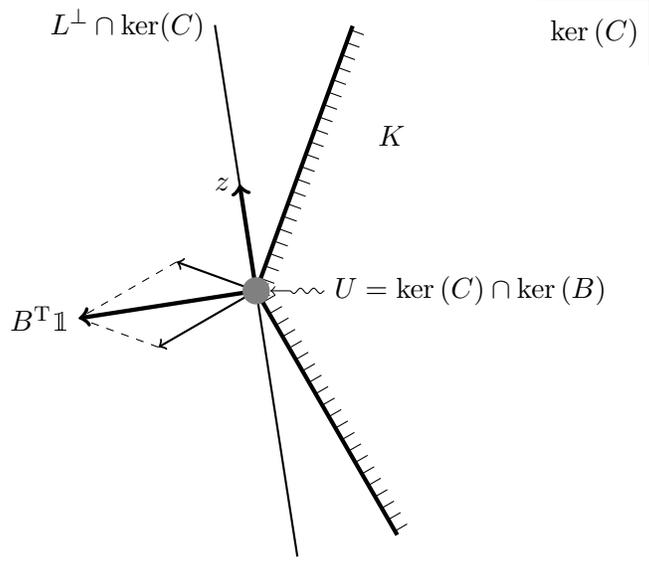

		\centering
		\input Fig-Lem1b
		\caption{Illustration of the proof of Part~(ii) of Lemma~\ref{lem:GeomKonv:polyKegel:HilfeVA}. Note that the vector labeled~$\transpose{B}\oneVec{}$ is  the orthogonal projection of~$\transpose{B}\oneVec{}$ to $\ker{(C)}$.}
		\label{fig:Lem1b}
	\end{figure}
	Due to~$\dim(L)\le 1$, the orthogonal complement $L^{\perp}\subseteq\R^n$ of~$L$ in~$\R^n$ has  at least dimension~$n-1$, hence,
	\begin{equation}\label{eq:proofLemmaDimLorth}
		\dim(L^{\perp}\cap\kernel{C})\ge\dim(\kernel{C})-1
	\end{equation}
	holds. If the linear subspace  $L^{\perp}\cap\kernel{C}$ was contained in~$K$, then (due to $K\subseteq\setDef{x\in\R^n}{Bx\le\zeroVec{p}}$)
	 the linear subspace 
	 $L^{\perp}\cap\kernel{C}$  would be a subset of $\kernel{B}$ (again, as $Bx\le\zeroVec{}$ and $-Bx\le\zeroVec{}$ imply $x\in\kernel{B}$), which, by~\eqref{eq:proofLemmaDimLorth}, would contradict the assumption on the dimension of $\dim(\kernel{B}\cap\kernel{C})$ in the second part of the lemma. 

Thus, there is some	$z\in (L^{\perp}\cap\kernel{C})\setminus K$.
	Suppose we have $-z\in K$, thus $B(-z)\le\zeroVec{P}$, and hence $Bz\ge\zeroVec{p}$. Because of $\scalProd{\oneVec{p}}{Bz}=\scalProd{\transpose{B}\oneVec{p}}{z}=0$
	(due to $z\in L^{\perp}$) this implies $Bz=\zeroVec{p}$, which, however, due to $z\in\kernel{C}$, contradicts $z\not\in K$. Thus, $z$ is a vector as claimed to exist in the second part of the lemma.
\end{proof}

Using Lemma~\ref{lem:GeomKonv:polyKegel:HilfeVA}, we can now easily prove Theorem~\ref{thm:polyConeFinite} by establishing, by induction on $p=0,1,\dots$, that, for every $B\in\R^{p\times n}$ and $C\in\R^{q\times n}$ (with $p+q\ge 1$, $n\ge 1$) there is a finite set $X\subseteq \qsd{A}^n$  with
\[
K=\setDef{x\in\R^n}{Bx\le\zeroVec{p},Cx=\zeroVec{q}}=\ccone(X)
\]
(where $A=\small{\begin{pmatrix}B\\C\end{pmatrix}}\in\R^{(p+q)\times n}$). 

For $p=0$, this follows readily from the first part of  Lemma~\ref{lem:GeomKonv:polyKegel:HilfeVA}. 
For $p\ge 1$, we may assume that there is some
$z\in\kernel{C}\setminus\{\zeroVec{n}\}$
with
$z,-z\not\in K$ 
 (otherwise, the claim again follows by the first part of Lemma~\ref{lem:GeomKonv:polyKegel:HilfeVA}). 
For all $i\in\{1,\dots,p\}$, let $B^{(i)}\in\R^{(p-1)\times n}$ be the matrix that arises from~$B$ by deleting row~$i$.   
By induction hypothesis, applied to
\[
K_i:=\setDef{x\in\R^n}{B^{(i)}x\le\zeroVec{p-1},\scalProd{\row{B}{i}}{x}=0,Cx=\zeroVec{q}}\,,
\]
there is some finite set  $X_i\subseteq\qsd{A}^n$ with $K_i=\ccone(X_i)$. It suffices to show $K\subseteq\ccone(X)$ for $X=\bigcup_{i=1}^p X_i$
 (because $\ccone(X)\subseteq K$ is clear).
Towards this end, let $x\in K$ (see~Fig.~\ref{fig:Thm}). 
\begin{figure}[ht]
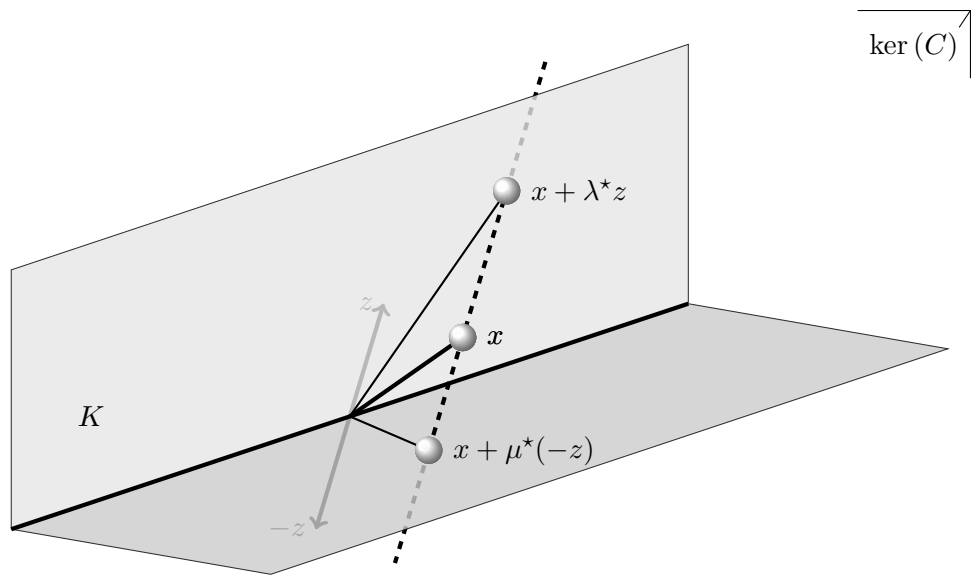

	\centering
	\input Fig-Thm
	\caption{Convex combination of~$x$ in the polyhedral cone~$K$ (in this example being the intersection of two linear halfspaces in $\ker{(C)}$, viewed ``from above'') by two vectors from $\ccone{(X)}$ in the proof of Theorem~\ref{thm:polyConeFinite}.}
	\label{fig:Thm}
\end{figure}
The set
\[
	I=\setDef{i\in\{1,\dots,p\}}{\scalProd{\row{B}{i}}{z}>0}\ne\varnothing
\] 
is non-empty due to $Cz=\zeroVec{}$ and $z\not\in K$. For all $i\in I$, the number 
$\lambda_i=-\frac{\scalProd{\row{B}{i}}{x}}{\scalProd{\row{B}{i}}{z}}\ge 0$
is nonnegative
(due to $Bx\le\zeroVec{}$). We have $\scalProd{\row{B}{i}}{(x+\lambda z)}\le 0$ for all $0\le \lambda\le\lambda_i$ with equality for $\lambda=\lambda_i$. Now we choose $i\st\in I$  such that 
$\lambda\st=\lambda_{i\st}=\min\setDef{\lambda_i}{i\in I}$
holds (with $\lambda\st\ge 0$). Then, we have $B(x+\lambda\st z)\le\zeroVec{p}$ and $\scalProd{\row{B}{i\st}}{(x+\lambda\st z)}=0$. Thus, (due to $Cz=\zeroVec{q}$) we conclude 
$x+\lambda\st z\in K_{i\st}= \ccone(X_{i\st})\subseteq\ccone(X)$.

Similarly, (due to $-z\not\in K$) one finds some $\mu\st\ge 0$ with $x+\mu\st (-z)\in \ccone{X}$. Hence~$x$, as a convex combination of $x+\lambda\st z\in\ccone{X}$ and $x-\mu\st z\in\ccone{X}$ (with $\lambda\st,\mu\st\ge 0$), is contained in~$\ccone(X)$. This concludes the proof of  Theorem~\ref{thm:polyConeFinite}.

\paragraph*{Acknowledgments} I am grateful to Matthias Peinhardt for valuable comments on a draft of this note.

\bigskip

\end{document}

%% file: Fig-Lem1a.tex
\begin{tikzpicture}[scale=1.5]
 	\coordinate [label=below:$\ker{(C)}$] (kerC) at (5cm,4cm);
	\draw ($(kerC)+(-0.5,0.1)$) -- ($(kerC)+(0.5,0.1)$) -- ($(kerC)+(0.5,-0.5)$); 
	\coordinate [label=below:$\zeroVec{}$] (orig) at (0,0);
 	\coordinate (UAnfang) at (canvas polar cs:angle=190,  radius=3cm) ;
	\coordinate (UEnde) at (canvas polar cs:angle=10,  radius=4cm) ;
	\coordinate (yEnde) at (canvas polar cs:angle=70, radius=4cm) ;
	\coordinate (aEnde) at (canvas polar cs:angle=100, radius=2cm) ;
	\coordinate (yaaEnde) at ($(orig)!(yEnde)!(aEnde)$) ;
	\coordinate (uEnde) at ($(orig)!(yEnde)!(UEnde)$) ;
	\coordinate [label=above:${U=\ker{(C)}\cap\ker{(B)}}$] (U) at ($(UAnfang)!1.0!(UEnde)$);
	\coordinate [label=right:$y$]  (y) at ($(orig)!1.0!(yEnde)$);
	\coordinate [label=left:$a$]  (a) at ($(orig)!1.0!(aEnde)$);
	\coordinate [label=left:$\scalProd{a}{y}a$]  (yaa) at ($(orig)!1.0!(yaaEnde)$);
	\coordinate [label=below:$u$]  (a) at ($(orig)!1.0!(uEnde)$);
 	\coordinate (aplusu) at ($(aEnde)+(uEnde)$) ;
	\coordinate [label=right:$\tfrac{1}{\scalProd{a}{y}}y$] (yayEnde) at (intersection of aEnde--aplusu and orig--yEnde);
	\coordinate [label=below:$u'$] (u'Ende) at ($(aEnde)-(yayEnde)$); 
	\draw [-,thick] (UAnfang) to (UEnde);
	\draw [->,ultra thick] (orig) to (yEnde);
	\draw [->,ultra thick] (orig) to (aEnde);
	\draw [->,very thick] (orig) to (yaaEnde);
	\draw [->,ultra thick] (orig) to (uEnde);
	\draw [->,ultra thick] (orig) to (yayEnde);
	\draw [->,ultra thick] (orig) to (u'Ende);
	\draw [-,dashed] (yaaEnde) to (yEnde);
	\draw [-,dashed] (uEnde) to (yEnde);
	\draw [-,dashed] (u'Ende) to (aEnde);
	\draw [-,dashed] (yayEnde) to (aEnde);
\end{tikzpicture}

%% file: Fig-Lem1b.tex
\begin{tikzpicture}[scale=1.5]
 	\coordinate [label=below:$\ker{(C)}$] (kerC) at (3cm,2.5cm);
	\draw ($(kerC)+(-0.5,0.1)$) -- ($(kerC)+(0.5,0.1)$) -- ($(kerC)+(0.5,-0.5)$); 
 	\coordinate (ray1Ende) at (canvas polar cs:angle=70,  radius=2.5cm);
 	\coordinate (normalVec1Ende) at ($(orig)!.3!90:(ray1Ende)$);
 	\coordinate (ray2Ende) at (canvas polar cs:angle=300,  radius=2.5cm);
 	\coordinate (normalVec2Ende) at ($(orig)!.4!270:(ray2Ende)$);
	\coordinate [label=left:$\transpose{B}\oneVec{}$] (BT1Ende) at ($(normalVec1Ende)+(normalVec2Ende)$);
	\coordinate [label=left:$L^{\perp}\cap\kernel{C}$] (LOrthEnde) at ($(orig)!1.5!270:(BT1Ende)$);
	\coordinate [label=left:$z$] (zEnde) at ($(orig)!.4!(LOrthEnde)$);
	\coordinate [label=above:$K$] (K) at (1.2,1.2);
	\draw [-,ultra thick] (orig) to (ray1Ende);
	\draw [-,decorate,decoration={ticks,raise=-3pt,segment length=5pt}] (orig) to (ray1Ende);
	\draw [-,ultra thick] (orig) to (ray2Ende);
	\draw [-,decorate,decoration={ticks,raise=3pt,segment length=5pt}] (orig) to (ray2Ende);
	\draw [->,thick] (orig) to (normalVec1Ende);
	\draw [->,thick] (orig) to (normalVec2Ende);
	\draw [->,ultra thick] (orig) to (BT1Ende);
	\draw [-,thick] ($(orig)!-1!(LOrthEnde)$) to (LOrthEnde);
	\draw [->,ultra thick] (orig) to (zEnde);
	\draw [-,dashed] (normalVec1Ende) to (BT1Ende);
	\draw [-,dashed] (normalVec2Ende) to (BT1Ende);
	\node [circle,fill=black!50] (U) at (0,0) {};
	\node (ULabel) at (1.9,0) {${U=\ker{(C)}\cap\ker{(B)}}$};
	\draw [->,decorate,decoration={snake,segment length=5pt,amplitude=1pt,post length=5pt}] (ULabel) to (U);

\end{tikzpicture}

%% file: Fig-Thm.tex
\begin{tikzpicture}[scale=1.5]
 	\coordinate [label=below:$\ker{(C)}$] (kerC) at (5cm,3.5cm);
	\draw ($(kerC)+(-0.5,0.1)$) -- ($(kerC)+(0.5,0.1)$) -- ($(kerC)+(0.5,-0.5)$); 
	\draw ($(kerC)+(0.5,0.1)$) -- ++(-.1,-.15); 	
	\coordinate (orig) at (0,0);
	\coordinate (UAnfang) at (-3,-1);
	\coordinate (UEnde) at (3,1);
	\coordinate [label=right:$K$] (K) at ($(UAnfang)+(.5,1)$);
	\coordinate (vecHinten) at (0,2.3);
	\coordinate (vecVorne) at (2.3,-0.4);
	\coordinate [label=left:$z$] (zEnde) at (0.3,1);
	\coordinate [label=left:$-z$] (-zEnde) at ($-1*(zEnde)$);
	\coordinate (xEnde) at (1.0,0.7) {};
	\coordinate (lochHinten) at ($(xEnde)+1.3*(zEnde)$) {};
	\coordinate (lochVorne) at ($(xEnde)-1.0*(zEnde)$) {};
	\draw [->,ultra thick] (orig) to (zEnde);
	\draw [->,ultra thick] (orig) to ($-1*(zEnde)$);
	\draw [-,dashed,ultra thick] (lochHinten) -- ++($1.2*(zEnde)$);
	\draw [-,dashed,ultra thick] (lochVorne) -- ++($1.0*(-zEnde)$);
	\draw [fill=black!10,opacity=.8] (UAnfang) -- ++(vecHinten) -- ++($(UEnde)-(UAnfang)$) -- (UEnde);
	\draw [fill=black!20,opacity=.8] (UAnfang) -- ++(vecVorne) -- ++($(UEnde)-(UAnfang)$) -- (UEnde);
	\draw [-,ultra thick] (UAnfang) to (UEnde);
	\coordinate [label=right:$K$] (K) at ($(UAnfang)+(.5,1)$);
	\node [circle,ball color=white,opacity=.3,label=right:$x$] (xEndeNode) at (xEnde) {};
	\draw [->,ultra thick] (orig) to (xEnde);
	\draw [->,thick] (orig) to (lochHinten);
	\draw [->, thick] (orig) to (lochVorne);
	\draw [-,dashed,ultra thick] (lochVorne) to (lochHinten);
	\node [circle,ball color=white,opacity=.5,label=right:$x$] (xEndeNode) at (xEnde) {};
	\node [circle,ball color=white,opacity=.5,label=right:$x+\lambda^{\star}z$] (lochHintenNode) at (lochHinten) {};
	\node [circle,ball color=white,opacity=.5,label=right:$x+\mu^{\star}(-z)$] (lochVorneNode) at (lochVorne) {};
	
\end{tikzpicture}